\documentclass[12pt]{article}
\usepackage{amsmath}
\usepackage{amsthm}
\usepackage{amsfonts}
\usepackage{amssymb}
\usepackage{graphicx}
\usepackage{tikz}
\usepackage{tikz-cd}
\usepackage{mathtools}
\usepackage{mathrsfs}
\usepackage{amsrefs}
\usepackage[none]{hyphenat}
\usepackage{hyperref}
\usepackage[capitalise, nameinlink]{cleveref-usedon}
\newtheorem{theorem}{Theorem}[section]
\newtheorem{corollary}[theorem]{Corollary}
\newtheorem{lemma}[theorem]{Lemma}
\newtheorem{prop}[theorem]{Proposition}

\newtheorem{remark}[theorem]{Remark}
\theoremstyle{definition}
\newtheorem{definition}[theorem]{Definition}

\title{An algebraic proof of Hrushovski's theorem}
\author{Thomas Wisson}
\date{}

\begin{document}
\maketitle

\begin{abstract}
In his paper on the Mordell-Lang conjecture, Hrushovski employed techniques from model theory to prove the function field version of the conjecture. In doing so he was able to answer a related question of Voloch, which we refer to henceforth as Hrushovski's theorem. In this paper we shall give an alternative proof of said theorem in the characteristic $p$ setting, but using purely algebro-geometric ideas.
\end{abstract}

\section{Introduction}
Let $U$ be a smooth curve over $k=\bar{\mathbb{F}}_{p}$, and let $K$ be its function field. Then take an abelian variety $A$ over $K$ with $K/k$-trace equal to zero, and a finitely generated subgroup $\Gamma\subseteq A(K)$. Given a discrete valuation $v$ on $K$, recall the notion of $v$-adic distance $d_{v}(Y,P)$ between a subvariety $Y$ of $A$, and a point $P$ on $A$. We can also define the local height $\lambda_{v}(Y,P):=-\text{log}~d_{v}(Y,P)$.
\begin{remark}
\label{rem:Subvar}
For us, a subvariety of $A$ can be any reduced closed subscheme of $A_{K'}$, where $K'$ is any algebraic field extension of $K$.
\end{remark}
\begin{definition}
\label{def:LinSub}
A subvariety $Y\subset A$ is said to be linear if it is a finite union of translates of abelian subvarieties of $A$.
\end{definition}
\begin{theorem}[Hrushovski's Theorem]
\label{th:Hrush}
Let $X\subset A$ be a subvariety defined over $K$. There exists a linear subvariety $Y\subset X$ which is also defined over $K$, such that for each discrete valuation $v$ of $K$ there is a constant $C_{v}$ for which
\begin{align*}
\lambda_{v}(X,P)\leq \lambda_{v}(Y,P)+C_{v}
\end{align*}
for all points $P\in\Gamma$.
\end{theorem}
Note that what Hrushovski originally proved in Theorem 6.4. of \cite{Hru} is slightly weaker than the result we prove here.
\begin{remark}
\label{rem:HrushDist}
In terms of the $v$-adic distance, the above inequality translates into \begin{align*}
d_{v}(Y,P)\leq e^{C_{v}}\cdot d_{v}(X,P)
\end{align*}
Of course, the $v$-adic distance is only defined up to a multiplicative constant, and so the constant in Hrushovski's theorem is unavoidable.
\end{remark}
\begin{remark}
\label{rem:TrZero}
If we remove the condition on the $K/k$-trace being zero we still can define a subvariety $Y\subset X$ satisfying the same inequality, however $Y$ may no longer be linear. Instead, $Y$ will be a finite union of ``special" subvarieties, see Theorem 1.~and the subsequent discussion in \cite{Hru}.
\end{remark}
\begin{remark}
\label{rem:Field}  
In Hrushovski's original proof $Y$ is defined only over some extension field of $K$. Poonen and Voloch showed in \cite{Poo-Vol} that one can always find such a $Y$ defined over $K$ by using an argument from model theory. Our method then reproves this fact, but without the model theoretic input.
\end{remark}
\begin{corollary}
\label{cor:HrushPnt}
We have an equality of sets
\begin{align*}
X(K)\cap\Gamma=Y(K)\cap\Gamma
\end{align*}
\end{corollary}
\begin{remark}
The corollary above implies the Mordell-Lang conjecture in the sense that if $X(K)\cap\Gamma$ is dense in $X$, then so must $Y(K)\cap\Gamma$ be, and hence $X=Y$ is linear. Thus demonstrating that Hrushovski's theorem can be thought of as a continuous version of Mordell-Lang.
\end{remark}
We can and shall assume throughout that $X$ is irreducible. Our approach to proving \cref{th:Hrush} is to consider the closed subschemes
\begin{align*}
\text{Exc}^{n}(A,X^{+Q})\subseteq X^{+Q}
\end{align*}
which we refer to generally as exceptional schemes; we will define them precisely in the next chapter. Here $n\in\mathbb{N}$ and $X^{+Q}$ denotes the translation of $X$ by a point $Q\in\Gamma$.\\

We have the following two theorems that will allow us to derive \cref{th:Hrush}.
\begin{theorem}
\label{th:ExcSub}
Assume that $X$ is not linear. Then there exists $m\in\mathbb{N}$ such that for all $Q\in\Gamma$ we have a strict inclusion
\begin{align*}
\textnormal{Exc}^{m}(A,X^{+Q})\subsetneq X^{+Q}
\end{align*}
\end{theorem}
\begin{theorem}
\label{th:ExcDist}
Let $P\in p^k\Gamma$ be such that $d_{v}(X,P)\leq 1/p^n$ with $n$ sufficiently large, then
\begin{align*}
d_{v}(\textnormal{Exc}^{k}(A,X),P)\leq 1/p^n
\end{align*}
\end{theorem}
We shall prove each of these theorems in the subsequent chapters. Let us now show how to prove Hrushovski's theorem using these results.
\begin{proof}[Proof. (of \cref{th:Hrush})]
If $X$ is linear, then we are done. So assume otherwise, then by \cref{th:ExcSub} we can choose $m\in\mathbb{N}$ such that
\begin{align*}
\text{Exc}^{m}(A,X^{+Q})\subsetneq X^{+Q}
\end{align*}
for all $Q\in\Gamma$. In particular, since $\Gamma$ is finitely generated, we can take $\lbrace Q_{i}\rbrace$ to be a set of representatives of the equivalence classes of the finite group $\Gamma/p^m\Gamma$, and define
\begin{align*}
Y_{1}:=\cup_{i}\text{Exc}^{m}(A,X^{+Q_{i}})^{-Q_{i}}\subsetneq X
\end{align*}
Although $Y_{1}$ may still not be linear, we can repeat this construction with each irreducible component of $Y_{1}$ in place of $X$ to get $Y_{2}\subseteq Y_{1}$. Inductively, this gives a descending chain of closed subschemes
\begin{align*}
\cdots\subseteq Y_{k}\subseteq\cdots\subseteq Y_{2}\subseteq Y_{1}\subseteq X
\end{align*}
which by noetherianity must stabilize at some $Y_{N+1}=Y_{N}$. By \cref{th:ExcSub} $Y:=Y_{N}$ then has to be linear.\\

Now, if $P\in\Gamma$ is such that $d_{v}(X,P)\leq 1/p^n$, then we can choose $Q_{i}$ such that $P+Q_{i}\in p^m\Gamma$, and we will have
\begin{align*}
d_{v}(X^{+Q_{i}},P+Q_{i})\leq 1/p^n
\end{align*}
therefore by \cref{th:ExcDist} we get that
\begin{align*}
d_{v}(\text{Exc}^{m}(A,X^{+Q_{i}}),P+Q_{i})\leq 1/p^n &\implies d_{v}(\text{Exc}^{m}(A,X^{+Q_{i}})^{-Q_{i}},P)\leq 1/p^n\\
&\implies d_{v}(Y_{1},P)\leq 1/p^n
\end{align*}
Repeating this argument inductively to get that $d_{v}(Y_{k},P)\leq 1/p^n$ for all $k$, which of course implies $d_{v}(Y,P)\leq 1/p^n$, this completes the proof.

\end{proof}
\section{The Exceptional Schemes}
Take the N\'eron model $\mathcal{A}$ of $A$ over $U$, and assume $U$ is sufficiently small so that the closed immersion $X\hookrightarrow A$ extends to a closed immersion $\mathcal{X}\hookrightarrow\mathcal{A}$, with $\mathcal{X}$ flat over $U$. We do not however consider $U$ to be fixed, as we will need to vary it throughout the paper.\\

Note that for the N\'eron model there is a bijective correspondence between $A(K)$ and $\mathcal{A}(U)$, therefore we shall seldom distinguish between a point $Q\in A(K)$ and its lifting $Q\in\mathcal{A}(U)$.\\

In \cite{Ros} R\"ossler defines the exceptional schemes $\text{Exc}^{n}(\mathcal{A},\mathcal{X})$ over $U$ for the purposes of proving the Mordell-Lang conjecture. We shall go over the construction and define the analogous objects $\text{Exc}^{n}(A,X)$ over $K$, as-well as of course proving they are compatible, i.e.
\begin{align*}
\text{Exc}^{n}(\mathcal{A},\mathcal{X})_{K}=\text{Exc}^{n}(A,X)
\end{align*}
\subsection{Defining the Exceptional Schemes}
Let us first introduce the notion of the Weil restriction functor. Start with a scheme $T$ and a morphism $T'\rightarrow T$. Then for each $T'$-scheme $Z$ we can consider the functor
\begin{align*}
W/T\mapsto \text{Hom}_{T'}(W\times_{T}T',Z)
\end{align*}
If $T'$ is finite, flat and locally of finite presentation over $T$ then by 7.6. in \cite{Bos} this functor is representable by a $T$-scheme which we denote $\mathfrak{R}_{T'/T}(Z)$.\\

Now, consider the diagonal immersion $\Delta:U\rightarrow U\times_{k} U$. Let $\mathcal{I}_{\Delta}\subseteq\mathcal{O}_{U\times U}$ be the ideal sheaf of $\Delta_{*}U$. For each $n\geq 0$ define $U_{n}$, the $n$-th infinitesimal neighbourhood of the diagonal inside $U\times_{k} U$, as the closed subscheme associated to $\mathcal{O}_{U\times U}/\mathcal{I}_{\Delta}^{n+1}$.\\

From the two projection maps $\pi_{1},\pi_{2}:U\times_{k} U\rightarrow U$, we obtain the induced maps $\pi^{U_{n}}_{1},\pi^{U_{n}}_{2}:U_{n}\rightarrow U$. We view $U_{n}$ as a $U$-scheme via $\pi^{U_{n}}_{1}$.
\begin{lemma}
\label{lem:Fin-Flat}
$U_{n}$ is finite and flat as a $U$-scheme.
\end{lemma}
\begin{proof}
See Lemma 2.1. in \cite{Ros}.

\end{proof}
This allows us to make the following definition for a scheme $W$ over $U$.
\begin{definition}
\label{def:JetSch}
The $n$-th jet scheme of $W$ over $U$ is defined as
\begin{align*}
J^{n}(W/U):=\mathfrak{R}_{U_{n}/U}(\pi^{U_{n},*}_{2}W)
\end{align*}
\end{definition}
For each $m\leq n$ there are morphisms $U_{m}\rightarrow U_{n}$, and these subsequently induce morphisms $\Lambda^{W}_{n,m}:J^{n}(W/U)\rightarrow J^{m}(W/U)$ of the jet schemes. Furthermore, there is a map of sets
\begin{align*}
\lambda^{W}_{n}:W(U)\rightarrow J^{n}(W/U)(U)
\end{align*}
which sends $f:U\rightarrow W$ to $J^{n}(f):U=J^{n}(U/U)\rightarrow J^{n}(W/U)$.
\begin{lemma}
\label{lem:lambda}
We have the following identity for all $m\leq n$
\begin{align*}
\Lambda^{W}_{n,m}\circ\lambda^{W}_{n}=\lambda^{W}_{m}
\end{align*}
\end{lemma}
\begin{proof}
This follows from the commutative diagram for any $f:U\rightarrow W$
\begin{center}
\begin{tikzcd}
J^{n}(U/U) \arrow[r, "\Lambda^{U}_{n,m}=\text{id}_{U}"] \arrow[d, "J^{n}(f)"']
& J^{m}(U/U) \arrow[d, "J^{m}(f)"]\\
J^{n}(W/U) \arrow[r, "\Lambda^{W}_{n,m}"']
& J^{m}(W/U)
\end{tikzcd}
\end{center}

\end{proof}
\begin{lemma}
\label{lem:lam-func}
For any $U$-morphism $g:W\rightarrow W'$ we have
\begin{align*}
J^{n}(g)\circ\lambda^{W}_{n}=\lambda^{W'}_{n}\circ g
\end{align*}
\end{lemma}
\begin{proof}
This simply follows from the fact that $J^{n}$ is a functor and so
\begin{align*}
J^{n}(g\circ f)=J^{n}(g)\circ J^{n}(f)
\end{align*}
for any $f:U\rightarrow W$.

\end{proof}
Next, let $K\rightarrow U$ be the generic point of $U$. We can consider the analogous diagonal immersion $\Delta_{K}:K\rightarrow K\times_{k}K$, and the associated ideal $I_{\Delta}\subseteq\mathcal{O}_{K\times K}$, in order to define $K_{n}:=\mathcal{O}_{K\times K}/I^{n+1}_{\Delta}$.\\

As before, we also have maps $\pi^{K_{n}}_{1},\pi^{K_{n}}_{2}:K_{n}\rightarrow K$, and we view $K_{n}$ as a $K$-scheme via $\pi^{K_{n}}_{1}$.
\begin{lemma}
\label{lem:BaseChg}
We have that
\begin{align*}
K_{n}=U_{n}\times_{U}K
\end{align*}
\end{lemma}
\begin{proof}
We can assume wlog that $U=\text{Spec}(R)$, in which case we need to show that
\begin{align*}
(R\otimes_{k}R/\mathcal{I}^{n+1}_{\Delta})\otimes_{R}K=K\otimes_{k}K/I^{n+1}_{\Delta}
\end{align*}
where
\begin{align*}
\mathcal{I}_{\Delta}=\langle r\otimes 1-1\otimes r:r\in R\rangle\\
I_{\Delta}=\langle x\otimes 1-1\otimes x:x\in K\rangle
\end{align*}
First of all, we have a natural inclusion
\begin{align*}
\phi:(R\otimes_{k}R/\mathcal{I}^{n+1}_{\Delta})\otimes_{R}K&\hookrightarrow K\otimes_{k}K/I^{n+1}_{\Delta}\\
r\otimes r'\otimes x&\mapsto rx\otimes r'
\end{align*}
where we have used the fact that $I^{n+1}_{\Delta}\cap R\otimes_{k}R=\mathcal{I}^{n+1}_{\Delta}$.\\

To get surjectivity, take any $r\in R^{\times}$ and note that
\begin{align*}
(r\otimes 1\otimes 1)\cdot(1\otimes 1\otimes r^{-1})=r\otimes 1\otimes r^{-1}=1\otimes 1\otimes rr^{-1}=1
\end{align*}
and thus $r\otimes 1\otimes 1$ is invertible. Furthermore, since $r\otimes 1-1\otimes r\in \mathcal{I}_{\Delta}$, the element $r\otimes 1\otimes 1-1\otimes r\otimes 1$ is nilpotent.\\

Since the sum of an invertible element and a nilpotent one is also invertible:
\begin{align*}
\frac{1}{i+n}=\frac{i^{-1}}{1+i^{-1}n}=i^{-1}(1-i^{-1}n+\cdots+(-i^{-1}n)^k)
\end{align*}
it follows that $1\otimes r\otimes 1=r\otimes 1\otimes 1-(r\otimes 1\otimes 1-1\otimes r\otimes 1)$ is also invertible.\\

The image under $\phi$ of these inverses must be $r^{-1}\otimes 1$ and $1\otimes r^{-1}$ respectively. Hence, $\phi$ is surjective and so an isomorphism.

\end{proof}
By the above lemma, $K_{n}$ is also finite and flat as a $K$-scheme, so we can also define the jet schemes for a scheme $W$ now over $K$.
\begin{definition}
\label{def:GnrJetSch}
The $n$-th jet scheme of $W$ over $K$ is
\begin{align*}
J^{n}(W/K):=\mathfrak{R}_{K_{n}/K}(\pi^{K_{n},*}_{2}W)
\end{align*}
\end{definition}
Let us check that these two definitions are compatible.
\begin{lemma}
\label{lem:JetBaseChg}
We have that
\begin{align*}
J^{n}(W/U)_{K}=J^{n}(W_{K}/K)
\end{align*}
\begin{proof}
Since the Weil restriction functor is compatible with base change (see again 7.6. in \cite{Bos}), we have that
\begin{align*}
J^{n}(W/U)_{K}&=\mathfrak{R}_{U_{n}/U}(\pi^{U_{n},*}_{2}W)_{K}\\
&=\mathfrak{R}_{K_{n}/K}((\pi^{U_{n},*}_{2}W)_{K_{n}})\\
&=\mathfrak{R}_{K_{n}/K}(\pi^{K_{n},*}_{2}W_{K})=J^{n}(W_{K}/K)
\end{align*}
where we have used that $U_{n}\times_{U}K=K_{n}$, as-well as the following Cartesian diagram
\begin{center}
\begin{tikzcd}[row sep=5, column sep=5]
& \pi^{K_{n},*}_{2}W_{K} \arrow[dl] \arrow[rr] \arrow[dd] & & W_{K}\arrow[dl] \arrow[dd] \\
\pi^{U_{n},*}_{2}W \arrow[rr, crossing over] \arrow[dd] & & W \\
& K_{n} \arrow[dl] \arrow[rr] & & K \arrow[dl] \\
U_{n} \arrow[rr] & & U \arrow[from=uu, crossing over]\\
\end{tikzcd}
\end{center}

\end{proof}
\end{lemma}
We must now impose the additional condition that $U$ is sufficiently small so that $\mathcal{A}$ is an abelian scheme over $U$. Then we can make the following definition.
\begin{definition}
\label{def:CritSch}
The $n$-th critical scheme of $\mathcal{X}$ over $U$ is
\begin{align*}
\text{Crit}^{n}(\mathcal{A},\mathcal{X}):=[p^n]_{*}(J^{n}(\mathcal{A}/U))\cap J^{n}(\mathcal{X}/U)
\end{align*}
where $[p^n]_{*}(J^{n}(\mathcal{A}/U))$ is the scheme theoretic image of $J^{n}(\mathcal{A}/U)$ by $[p^n]_{J^{n}(\mathcal{A}/U)}$.
\end{definition}
\begin{remark}
\label{rem:CritSch}
Note that because $\mathcal{A}$ is proper over $U$, it follows that $[p^n]_{*}(J^{n}(\mathcal{A}/U))$ is closed inside $J^{n}(\mathcal{A}/U)$, and also that $[p^n]_{*}(J^{n}(\mathcal{A}/U))\rightarrow\mathcal{A}$ is finite.
\end{remark}
There are maps
\begin{align*}
\cdots\rightarrow\text{Crit}^{2}(\mathcal{A},\mathcal{X})\rightarrow\text{Crit}^{1}(\mathcal{A},\mathcal{X})\rightarrow\mathcal{X}
\end{align*}
with each morphism finite. Hence we can finally define the exceptional schemes.
\begin{definition}
\label{def:ExcSch}
The $n$-th exceptional scheme $\text{Exc}^{n}(\mathcal{A},\mathcal{X})\subseteq\mathcal{X}$ over $U$ is the scheme theoretic image of the the morphism $ \text{Crit}^{n}(\mathcal{A},\mathcal{X})\rightarrow\mathcal{X}$.
\end{definition}
\begin{remark}
\label{rem:GnrExc}
Likewise, we can carry out this construction generically over $K$ in order to define $\textnormal{Crit}^{n}(A,X)$ and $\textnormal{Exc}^{n}(A,X)$.
\begin{prop}
\label{prop:ExcBaseChg}
We have that
\begin{align*}
\textnormal{Exc}^{n}(\mathcal{A},\mathcal{X})_{K}=\textnormal{Exc}^{n}(A,X)
\end{align*}
\end{prop}
\begin{proof}
Noting that taking scheme theoretic image commutes with flat base change, the pullback $\text{Exc}^{n}(\mathcal{A},\mathcal{X})_{K}$ is equal to the scheme theoretic image of the morphism $\text{Crit}^{n}(\mathcal{A},\mathcal{X})_{K}\rightarrow X$, and so it suffices to compute
\begin{align*}
\text{Crit}^{n}(\mathcal{A},\mathcal{X})_{K}&=[p^n]_{*}(J^{n}(\mathcal{A}/U))_{K}\cap J^{n}(\mathcal{X}/U)_{K}\\
&=[p^n]_{*}(J^{n}(A/K))\cap J^{n}(X/K)\\
&=\text{Crit}^{n}(A,X)
\end{align*}
where we have used \cref{lem:JetBaseChg} in the second line.

\end{proof}
\end{remark}
Note that the exceptional schemes give us a descending chain of closed subschemes
\begin{align*}
\cdots\subseteq\text{Exc}^{n}(\mathcal{A},\mathcal{X})\subseteq\cdots\subseteq\text{Exc}^{1}(\mathcal{A},\mathcal{X})\subseteq\mathcal{X}
\end{align*}
By noetherianity this chain must stabilize at some $\text{Exc}^{N}(\mathcal{A},\mathcal{X})$. We set 
\begin{align*}
\text{Exc}(\mathcal{A},\mathcal{X}):=\cap_{n}\text{Exc}^{n}(\mathcal{A},\mathcal{X})=\text{Exc}^{N}(\mathcal{A},\mathcal{X})
\end{align*}
to be the intersection of all the exceptional schemes.
\section{Proving \cref{th:ExcSub}}
For this chapter we keep $U$ as it was previously, thereby allowing us to define the exceptional schemes over $U$.
\begin{prop}
\label{prop:ExcSub}
Assume that $X$ is not linear. Then there exists $m\in\mathbb{N}$ such that for all $Q\in\Gamma$ we have a strict inclusion
\begin{align*}
\textnormal{Exc}^{m}(\mathcal{A},\mathcal{X}^{+Q})\subsetneq \mathcal{X}^{+Q}
\end{align*}
\end{prop}
\begin{proof}
Since $X$ is assumed not to be linear, by Theorem 1.2. in \cite{Ros} this implies that for any field extension $L/K$ and any $Q\in A(L)$ the set $X^{+Q}_{L}\cap\text{Tor}(A(L))$ is not Zariski dense in $X^{+Q}_{L}$, where here we use the assumption that $A$ has $K/k$-trace zero. Hence the conditions of Theorem 3.1. in \cite{Ros} are met, which implies the proposition.

\end{proof}
\begin{proof}[Proof. (of \cref{th:ExcSub}).]
By \cref{prop:ExcSub} we have a strict containment over $U$:
\begin{align*}
\text{Exc}^{m}(\mathcal{A},\mathcal{X}^{+Q})\subsetneq \mathcal{X}^{+Q}
\end{align*}
and so if we set $V:=\mathcal{X}^{+Q}\setminus\text{Exc}^{m}(\mathcal{A},\mathcal{X}^{+Q})$, this is a non-empty open subscheme over $U$. Since $\mathcal{X}^{+Q}$ is flat over $U$, it follows that $V$ is too. Hence, the generic fibre $V_{K}$ must be non-empty, which implies
\begin{align*}
\text{Exc}^{m}(A,X^{+Q})\subsetneq X^{+Q}
\end{align*}
This completes the proof.

\end{proof}
\section{Proving \cref{th:ExcDist}}
For the duration of this chapter we suppose that $U$ contains the closed point $u$ corresponding to the discrete valuation $v$. This means that $\mathcal{A}$ may no longer be an abelian scheme, and hence we only have access to the the jet schemes over $U$, not the critical or exceptional schemes.
\begin{proof}[Proof. (of \cref{th:ExcDist})]
Since $d_{v}(X,P)\leq 1/p^n$, we have that $P\in \mathcal{X}(u_{n})$, where $u_{n}$ is the $n$-th infinitesimal neighbourhood of $u$.\\

Now consider $\widetilde{P}=\lambda^{A}_{k}(P)\in p^k J^{k}(A/K)(K)=[p^k]_{*}J^{k}(A/K)(K)$. If we apply $J^{k}(-/U)$ to the commutative diagram
\begin{center}
\begin{tikzcd}
u_{n} \arrow[r] \arrow[d]
& \mathcal{X} \arrow[hookrightarrow]{d}\\
U \arrow[r, "P"']
& \mathcal{A}
\end{tikzcd}
\end{center}
we get the diagram
\begin{center}
\begin{tikzcd}
u_{n} \arrow[r] \arrow[d]
& J^{k}(u_{n}/U) \arrow[r] \arrow[d]
& J^{k}(\mathcal{X}/U) \arrow[hookrightarrow]{d}\\
U \arrow[equal]{r} \arrow[bend right=25, "\widetilde{P}"']{rr}
& J^{k}(U/U) \arrow[r]
& J^{k}(\mathcal{A}/U)
\end{tikzcd}
\end{center}
where we have used the fact that
\begin{align*}
J^{k}(W/U)(u_{n})&=\text{Hom}_{U_{k}}(u_{n}\times_{U}U_{k},\pi^{U_{k},*}_{2}W)\\
&=\text{Hom}_{U_{k}}(u_{n}\times_{u}(u\times_{U}U_{k}),\pi^{U_{k},*}_{2}W)\\
&=\text{Hom}_{U_{k}}(u_{n}\times_{u}u_{k},\pi^{U_{k},*}_{2}W)\\
&=\text{Hom}_{u_{k}}(u_{n}\times_{u}u_{k},W_{u_{k}})\\
&=W(u_{n}\times_{u}u_{k})
\end{align*}
to define the map $u_{n}\rightarrow J^{k}(u_{n}/U)$ given by the projection $u_{n}\times_{u}u_{k}\rightarrow u_{n}$. Note also that the commutativity of the left hand square follows because $J^{k}(U/U)=U$ is terminal. Hence we have that $\widetilde{P}\in J^{k}(\mathcal{X}/U)(u_{n})$.\\

Thus this tells us, at least for $n$ sufficiently large, that
\begin{gather*}
d_{v}([p^n]_{*}J^{k}(A/K),\widetilde{P})=0\leq 1/p^{n}~~~~~d_{v}(J^{k}(X/K),\widetilde{P})\leq 1/p^{n}\\
\implies d_{v}(\text{Crit}^{k}(A,X),\widetilde{P})\leq 1/p^n
\end{gather*}
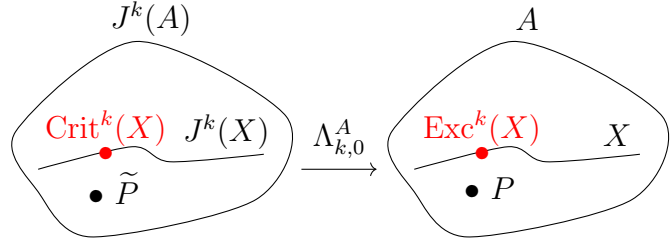
\begin{center}
\begin{tikzpicture}
\draw (0, 2) node{$J^{k}(A)$};
\draw (1, 0.5) node{$J^{k}(X)$};
\draw (-0.5, -0.25) node{$\bullet~\widetilde{P}$};
\draw plot [smooth cycle] coordinates {(1.4, -0.4) (-0.8, -0.9) (-1.5, -0.3) (-1.8, 0.5) (-0.2, 1.7) (1.8, 0.7)};
\draw plot [smooth] coordinates {(-1.5, 0) (-0.25, 0.3) (0.25, 0.1) (1.5, 0.2)};
\draw [red] (-0.6, 0.2) node{$\bullet$};
\draw [red] (-0.6, 0.55) node{$\text{Crit}^{k}(X)$};
\begin{scope}[shift={(5, 0)}]
\draw (0, 2) node{$A$};
\draw (1.2, 0.45) node{$X$};
\draw (-0.5, -0.25) node{$\bullet~P$};
\draw plot [smooth cycle] coordinates {(1.4, -0.4) (-0.8, -0.9) (-1.5, -0.3) (-1.8, 0.5) (-0.2, 1.7) (1.8, 0.7)};
\draw plot [smooth] coordinates {(-1.5, 0) (-0.25, 0.3) (0.25, 0.1) (1.5, 0.2)};
\draw [red] (-0.6, 0.2) node{$\bullet$};
\draw [red] (-0.6, 0.55) node{$\text{Exc}^{k}(X)$};
\end{scope}
\draw [->] (2, 0) -- (3, 0) node[pos=0.5, above]{$\Lambda^{A}_{k,0}$};
\end{tikzpicture}
\end{center}
since by definition $\text{Crit}^{k}(A,X)=[p^k]_{*}J^{k}(A/K)\cap J^{k}(X/K)$. So we can conclude that, since $\text{Crit}^{k}(A,X)\subseteq \Lambda^{A*}_{k,0}\text{Exc}^{k}(A,X)$, we have 
\begin{align*}
d_{v}(\text{Exc}^{k}(A,X),P)&=d_{v}(\Lambda^{A*}_{k,0}\text{Exc}^{k}(A,X),\widetilde{P})\\
&\leq d_{v}(\text{Crit}^{k}(A,X),\widetilde{P})\leq 1/p^n
\end{align*}

\end{proof}

\end{document}